\documentclass[11pt]{article}
\usepackage{amssymb}
\usepackage{mathrsfs}
\usepackage{latexsym,amsmath,amssymb,amsfonts,epsfig,graphicx,cite,psfrag}
\usepackage{color,colordvi,amscd}
\usepackage{verbatim}

\newcommand{\qed}{\hfill $\Box $}
\newcommand{\pf}{\noindent {\bf Proof.} }

\usepackage{amsthm}

\newcommand{\N}{{\mathbb{N}}}

\newcommand{\E}{{\mathbb{E}}}
\newcommand{\Z}{{\mathbb{Z}}}
\renewcommand{\P}{{\mathbb{P}}}

\renewcommand{\S}{{\mathcal{S}}}

\newcommand{\F}{{\mathcal{F}}}
\newcommand{\T}{{\mathcal{T}}}
\newcommand{\I}{{\mathcal{I}}}

\newcommand{\bd}{{\mathrm{DEG}}}
\newcommand{\sd}{{\mathrm{deg}}}

\newtheorem{theorem}{Theorem}[section]
\newtheorem{lemma}[theorem]{Lemma}

\begin{document}

\title{Rainbow matchings for 3-uniform hypergraphs}
\author{Hongliang Lu\footnote{luhongliang@mail.xjtu.edu.cn; partially supported by the National Natural
Science Foundation of China under grant No.11871391 and
Fundamental Research Funds for the Central Universities}\\
School of Mathematics and Statistics\\
Xi'an Jiaotong University\\
Xi'an, Shaanxi 710049, China\\
\smallskip\\
Xingxing Yu\footnote{yu@math.gatech.edu; partially supported by NSF
  grant DMS-1600738} \ and Xiaofan Yuan\\
School of Mathematics\\
Georgia Institute of Technology\\
Atlanta, GA 30332, USA}
\date{}
\maketitle

\begin{abstract}
K\"{u}hn, Osthus,  and Treglown  and, independently,  Khan
proved that if
$H$ is a $3$-uniform hypergraph with $n$ vertices such that $n\in 3\Z$ and large,
and $\delta_1(H)>{n-1\choose 2}-{2n/3\choose 2}$, then $H$ contains a perfect matching.
In this paper, we show that for $n\in 3\Z$ sufficiently large, if $F_1, \ldots, F_{n/3}$ are 3-uniform hypergrapghs
with a common vertex set and
$\delta_1(F_i)>{n-1\choose 2}-{2n/3\choose 2}$ for $i\in [n/3]$, then
$\{F_1,\dots, F_{n/3}\}$ admits a rainbow matching, i.e., a matching consisting of one edge from each $F_i$.
This is done by converting the rainbow matching problem to a perfect matching problem in a special class of uniform hypergraphs.

\end{abstract}

\newpage

\section{Introduction}

For any positive integer $k$ and any set $S$, let $[k]:=\{1,\ldots, k\}$ and
${S\choose k}:=\{T\subseteq S: |T|=k\}$. A {\it hypergraph} $H$
consists of a vertex set $V(H)$ and an edge set $E(H)\subseteq
2^{V(H)}$, and we write $e(H):=|E(H)|$ and often identify $E(H)$ with $H$.
For a positive integer $k$, a hypergraph $H$ is said to be {\it $k$-uniform}
if $E(H)\subseteq {V(H)\choose k}$, and a $k$-uniform hypergraph is also
called a {\it $k$-graph}.

A \emph{matching} in a hypergraph $H$ is a set of pairwise disjoint
edges in $H$, and we use $\nu(H)$ to denote the maximum size of a
matching in $H$. The problem for finding maximum matchings  in hypergraphs is NP-hard,
even for 3-graphs \cite{Ka72}. It is of interest to find good sufficient
conditions that guarantee large matchings.

Erd\H{o}s~\cite{Erdos65} conjectured in 1965 that,
for positive integers $k,n,t$, if $H$ is a  $k$-graph on
$n$ vertices and  $\nu(H) <
t$ then  $e(H)\leq \max \left\{{kt-1\choose k}, {n\choose k}-{n-t+1\choose k}\right\}.$
This bound is tight because of the complete $k$-graph on $kt-1$ vertices and the $k$-graph on $n$ vertices in which every
edge intersects a fixed set of $t-1$ vertices. For recent
progress on this conjecture, see \cite{AHS12,AFH12,FLM,Fr13,Fr17,HLS,LM}. In particular,  Frankl
\cite{Fr13} proved that if  $n\geq (2t-1)k-(t-1)$ and $\nu(H)<t$ then $e(H)\le
{n\choose k}-{n-t+1\choose k}$. This result was further improved by Frankl
and Kupavskii \cite{FrK19+}.

There has been extensive
study on degree conditions for large matchings in uniform hypergraphs.
Let $H$ be a hypergraph and $T\subseteq V(H)$.
The {\it degree}
of $T$ in $H$, denoted by $d_H(T)$, is the number of edges in $H$
containing $T$. For any integer  $l\ge 0$, let $\delta_l(H):=
\min\{d_H(T): T\in {V(H)\choose l}\}$ denote the minimum
{\it $l$-degree} of $H$. Hence,  $\delta_0(H)=e(H)$.  Note that $\delta_1(H)$ is often called the minimum {\it
vertex} degree of $H$.

For integers $n,k,d$ satisfying $0 \leq d \leq k-1$ and $n\in k\Z$,
let $m_d(k,n)$ denote the minimum integer $m$ such that every
$k$-graph $G$ on $n$ vertices with $\delta (G)\ge m$
has a perfect matching.  K\"uhn, Osthus and Treglown
\cite{KOT13} and, independently,
Khan \cite{Kh13} determined $m_1(k,n)$  for $k=3$ and large $n$.
 Khan \cite{Kh16} also determined $m_1(k,n)$  for $k=4$ and large $n$. For $d=k-1$, $m_{k-1}(k,n)$ was determined for large $n$ by R\"odl,
Ruci\'nski and Szemer\'edi [20]. This was generalized by Treglown and
Zhao \cite{TZ13} to the range $k/2\le d\le k-1$, where they also determined the extremal
families.

There are  attempts to extend the above conjecture of  Erd\H{o}s to a family
of hypergraphs. Let $\mathcal{F} = \{F_1,\ldots, F_t\}$ be a family
of hypergraphs. A set of pairwise disjoint edges, one from
each $F_i$, is called a \emph{rainbow matching} for $\mathcal{F}$. (In this situation, we also say that ${\cal F}$ or
$\{F_1,\ldots, F_t\}$ {\it admits} a rainbow matching.) Aharoni and
 Howard \cite{AH} made the following conjecture, which first appeared
 in  Huang, Loh, and Sudakov \cite{HLS}:
Let $t$ be a positive integer and ${\cal F}=\{F_1,\ldots,F_t\}$ such that, for $i\in [t]$,
$F_i\subseteq {[n]\choose k}$ and
$e(F_i)> \max\left\{{kt-1\choose k}, {n\choose k}-{n-t+1\choose k}\right\};$
 then  ${\cal F}$ admits a rainbow matching.
Huang, Loh, and Sudakov \cite{HLS} showed that this conjecture  holds when $t< n/3k^2$.

 In this paper, we prove a degree version of the above conjecture for rainbow matchings, which extends the results of
  K\"uhn, Osthus, and Treglown
\cite{KOT13} and, independently, of
Khan \cite{Kh13} for 3-graphs to families of 3-graphs.

\begin{theorem}\label{rainbow-pm}
Let $n\in 3\mathbb{Z}$ be positive and sufficiently large and let ${\cal F}=\{F_1,\ldots,F_{n/3}\}$ be a family of $n$-vertex 3-graphs such that
$V(F_i)=V(F_1)$ for $i\in [n/3]$.  If $\delta_1(F_i)>{n-1\choose
  2}-{2n/3\choose 2}$ for $i\in [n/3]$, then ${\cal F}$ admits a rainbow matching.
\end{theorem}
The bound on $\delta_1(F_i)$ in Theorem~\ref{rainbow-pm} is sharp.
To see this, let $m\leq n/3$ and let $H(n,m)$ denote a 3-graph that is isomorphic to the 3-graph with vertex set  $[n]$ and edge set
\[
\left\{e\in {[n]\choose 3}  :  e\not\subseteq [m] \mbox{ and } e\cap [m]\neq \emptyset\right\}.
\]
Note that for $n\in 3\Z$, $\delta_1(H(n,n/3-1))={n-1\choose 2}-{2n/3\choose 2}$
and $H(n,n/3-1)$ has no perfect matching. Hence, the family
of $n/3$ copies of  $H(n,n/3-1)$ admits no rainbow matching.

To prove Theorem~\ref{rainbow-pm}, we convert this rainbow matching
problem to a perfect matching problem for a special class of
hypergraphs. For any integer $k\ge 2$, a $k$-graph $H$ is \emph{$(1,k-1)$-partite }if there exists a
partition of $V(H)$ into sets $V_1, V_2$ (called {\it partition  classes}) such that for any $e\in E(H)$, $|e\cap V_1|=1$ and $|e\cap V_2|=k-1$.
A $(1,k-1)$-partite $k$-graph with partition classes $V_1,V_2$ is
\emph{balanced} if $(k-1)|V_1|=|V_2|$

 Let  $n\in 3\Z$, $Q=\{v_1,\ldots,v_{n/3}\}$ be a set of vertices, and
 $\mathcal{F}=\{F_1,\ldots,F_{n/3}\}$ be a family of $n$-vertex 3-graphs with common
 vertex set $P$ that is disjoint from $Q$. We use  $H_{1,3}(\mathcal{F})$ to represent  the balanced
 $(1,3)$-partite 4-graph with patition classes $Q, P$
and edge set $\bigcup_{i=1}^{n/3}E_i$, where $E_i=\{e\cup
\{v_i\} :  e\in E(F_i)\}$ for $i\in [n/3]$. If, for some $H(n,n/3)$ on $V(F_i)$,  $F_i=H(n,n/3)$ for $i\in [n/3]$,
then we write $H_{1,3}(n,n/3)$ for $H_{1,3}(\mathcal{F})$. The following observations will be useful:
\begin{itemize}
   \item [$(i)$] $E(F_i)$ is the neighborhood of $v_i$ in
     $H_{1,3}(\F)$ for $i\in [n/3]$,
     and ${\cal F}$ admits a rainbow matching if, and only if,
     $H_{1,3}({\cal F})$ has a perfect matching.
   \item [$(ii)$] $e(F_i)\ge \frac{n}{3}\delta_1(F_i)$ for all $i\in [n/3]$, and $d_{H_{1,3}(\F)}(v)\ge \sum_{i=1}^{n/3}\delta_1(F_i)$ for $v\in P$.
     \item [$(iii)$]  $d_{H_{1,3}(\F)}(\{u,v\})\ge  {n-1\choose
       2}-{2n/3\choose 2}+1$  for all $u\in P$ and $v\in Q$
     if $\delta_1(F_i)\ge {n-1\choose 2}-{2n/3\choose 2}+1$ for $i\in
     [n/3]$.
  \item [$(iv)$]  $\delta_1(H_{1,3}({\cal F}))\geq
     \frac{n}{3}\left({n-1\choose 2}-{2n/3\choose 2}+1\right)$ if $d_{H_{1,3}(\F)}(\{u,v\})\ge  {n-1\choose
       2}-{2n/3\choose 2}+1$  for
     all $u\in P$ and $v\in Q$.

 \end{itemize}

By observations $(i)$ and $(iii)$, Theorem \ref{rainbow-pm} follows from the following result.

\begin{theorem}\label{prefect-(1,3)}
 Let $n\in 3\Z$ be a positive and sufficiently large, and let $H$ be a
 $(1,3)$-partite $4$-graph with partition classes $Q,P$ such that
 $3|Q|=|P|=n$. Suppose $d_{H}(\{u,v\})\ge  {n-1\choose
       2}-{2n/3\choose 2}+1$  for
     all $u\in P$ and $v\in Q$.
Then $H$ has a perfect matching.
\end{theorem}

To prove Theorem~\ref{prefect-(1,3)},  we take the usual approach by  considering
whether or not $H$ is close to some $H_{1,3}(n,n/3)$ on the same vertex set.
Given $\varepsilon>0$ and two $k$-graphs $H_1,H_2$ with
$V(H_1)=V(H_2)$, we say that $H_2$ is $\varepsilon$-close to $H_1$
 if $|E(H_1)\setminus E(H_2)|<\varepsilon |V(H_1)|^k$.

In Section 2, we prove Theorem~\ref{prefect-(1,3)} when $H$ is close to
some $H_{1,3}(n,n/3)$, using the structure of $H_{1,3}(n,n/3)$
to find a perfect matching in $H$ greedily. This is the extremal case, as $H_{1,3}(n,n/3)$ is an extremal graph for
Theorem~\ref{prefect-(1,3)}.

In the non-extremal case, $H$ is not close to  any $H_{1,3}(n,n/3)$ on $V(H)$.
We first find a small matching $M'$ in $H$ that can be used to ``absorb" small sets of
vertices, then find an almost perfect matching $M''$ in $H-V(M')$, and  finally use $M'$ to absorb $V(H)\setminus V(M'\cup M'')$.
A more detailed account is given in Section 6.

In Section 3, we prove an  absorbing lemma for $(1,3)$-partite $4$-graphs, which can be used to find the absorbing matching $M'$.
In Section 5, we find the almost perfect matching $M''$ in $H-V(M')$. For this,
we use the approach of Alon et al. in \cite{AFH12} to find random subgraphs with desired properties (including the existence of
perfect fractional matchings). However, we need to modify this
approach to make it work, which is done in Section 4.
First, we need the random subgraphs to be balanced. Second, in the non-extremal case,
the (1,3)-partite 4-graphs do not have large sparse sets; so we also need to
control the independence number of those random subgraphs and for this we use
the hypergraph container result of Balogh  {\it et al}. \cite{BMS15}.

\section{The extremal case}

In this section, we prove Theorem~\ref{prefect-(1,3)}  for the case when $H$ is close to some
$H_{1,3}(n,n/3)$ on $V(H)$. First, we prove a result on rainbow matchings for a small family of
hypergraphs, which will serve as induction basis for our proof.

\begin{lemma} \label{HZ-16}
Let $n,t, k$ be positive integers such that $n> 2k^4 t$. Let $F_i$, $i\in [t]$,
be $n$-vertex $k$-graphs with a common vertex set.
If $\delta_1(F_i)>{n-1\choose k-1}-{n-t \choose k-1}$ for $i\in [t]$ then $\{F_1, \ldots, F_t\}$ admits a rainbow matching.
\end{lemma}

\begin{proof}
	We apply induction on $t$. Note that the assertion is trivial when $t=1$. So assume $t>1$ and
the assertion  holds for $t-1$. Then, since
	\(
		\delta_1(F_i)>{n-1\choose k-1}-{n-t \choose k-1}>{n-1\choose k-1}-{n-(t-1) \choose k-1},
	\)
	 $\{F_1, \ldots, F_{t-1}\}$ admits a rainbow matching, say $M$.
	
	Suppose for a contradiction that $\{F_1, \ldots, F_t\}$ does not admit a rainbow matching. Then
    every edge of $F_t$ must intersect $M$. So there exists  $v\in V(M)$ such that $d_{F_t}(v)> e(F_t)/(kt)$. Note that
	\[
		\delta_1(F_t)>{n-1\choose k-1}-{n-t \choose k-1}>\binom{n-1}{k-1}\left(1-\left(1-\frac{k-1}{n-1}\right)^t\right)>\frac{t(k-1)}{2(n-1)}\binom{n-1}{k-1},
	\]
	since $n> 2k^4 t$.
	So we have
	\[
		d_{F_t}(v)>\frac{\delta_1(F_t)n/k}{kt}>\frac{t(k-1)n}{2(n-1)k^2t}\binom{n-1}{k-1}>\frac{1}{2k^2}\binom{n-1}{k-1}.
	\]
	
    Let $F'_i=F_i-v$ for $i\in [t-1]$. Since
	\[
		\delta_1(F'_i)\ge \delta_1(F_i)-\binom{n-2}{k-2}>{n-1\choose k-1}-{n-t \choose k-1}-\binom{n-2}{k-2}= \binom{n-2}{k-1}-\binom{n-t}{k-1},
	\]
	it follows from induction hypothesis that $\{F'_1, \ldots, F'_{t-1}\}$ admits a rainbow matching, say $M'$.
	
	Note that the number of edges in $F_t$ containing $v$ and intersecting $M'$ is at most
	\[
		k(t-1)\binom{n-2}{k-2}<\frac{1}{2k^2}\binom{n-1}{k-1}<d_{F_t}(v),
	\]
	as $n\ge 2k^4t$. Hence, $v$ is contained in some edge of $F_t-V(M')$, say $e$.
    Now $M'\cup \{e\}$ is a rainbow matching for $\{F_1, \ldots, F_t\}$, a contradiction.
\end{proof}

Next, we prove Theorem~\ref{prefect-(1,3)} for the case when, for every
vertex $v$, most of the edges of $H_{1,3}(n,n/3)$ containing $v$ also
lie in $H$. More precisely, given $\alpha>0$,  $H_{1,3}(n,n/3)$, and a
$(1,3)$-partite 4-graph $H$ with $V(H)=V(H_{1,3}(n,n/3)$, we say
that a vertex $v\in V (H)$ is {\it $\alpha$-good}
if $|N_{H_{1,3}(n,n/3)}(v)\setminus N_H(v)|\le \alpha n^3$. Otherwise
we say that $v$ is {\it $\alpha$-bad}.

\begin{lemma}\label{good}
Let $n$ be positive integer and $H$ be a balanced $(1,3)$-partite $4$-graph on $4n/3$ vertices,
and let $\alpha$ be a constant such that $0<\alpha <2^{-12}$.
If all vertices of $H$ are $\alpha$-good with respect to some $H_{1,3}(n,n/3)$ on $V(H)$,
then $H$ has a perfect matching.
\end{lemma}

\pf Let $Q,P$ be the partition classes of $H$ such that $V(H(n,n/3))=P$ and let
$U,W$ denote the partition classes of $H(n,n/3)$ with $|W|=n/3$ and all
vertices of $H$ are $\alpha$-good with respect to $H_{1,3}(n,n/3)$.
Thus $|Q|=|W|=n/3$ and $|U|=2n/3$.

Let $M$ be a matching in $H$ that only uses edges consisting of two
vertcies from $U$ and one vertex from each of $Q$ and $W$, and choose
such $M$ that $|M|$ is maximum.
Let $Q':=Q\setminus V(M)$, $U'=U\setminus V(M)$, and
$W'=W\setminus V(M)$.  Then $|U'|/2=|W'|=|Q'|$.

Note that $|M|\geq n/4$. For, otherwise,
$|U'|/2=|W'|=|Q'|=n/3-|M|>n/12$. Then, by the maximality of $M$, we
have, for any $u\in U'$,
\[
|N_{H_{1,3}(n,n/3)}(u)\setminus N_H(u)|\geq |Q'||W'|(|U'|-1)> n^3/12^3>\alpha n^3,
\]
a contradiction.

Now suppose $M$ is not a perfect matching in $H$. Then $Q',U',W'$ are
all non-empty. Let
$v\in Q'$, $u_1,u_2\in U'$ be distinct,  and $w\in W'$.

Let $\{e_1, e_2,e_3\}$ be an arbitrary set of three pairwise distinct edges from $M$. By the maximality of
$M$,  no matching of size 4 in $H$ is contained in
$e_1\cup e_2\cup e_3\cup \{v,w,u_1,u_2\}$  and uses only edges with
two vertices from $U$ and one vertex from each of $Q$ and $W$.
Hence, there exists $S\in E(H_{1,3}(n,n/3))\setminus E(H)$ such that
$S\subseteq e_1\cup e_2\cup e_3\cup \{v,w,u_1,u_2\} $, $|S\cap e_i|=1$ for $i\in [3]$, $|S\cap
\{v,w,u_1,u_2\}|=1$,  and $S$ has two vertices from $U$ and one vertex from each of $Q$ and $W$.

Note that there are ${m\choose 3}$ choices for $\{e_1,e_2,e_3\}$, which
result in distinct choices for $S$. So the number of edges in
$E(H_{1,3}(n, n/3))\setminus E(H)$ containing exactly one vertex from
$\{v,w,u_1,u_2\}$ is at least
 $$ {m\choose 3}\ge {n/4\choose 3}>n^3/(2^{10}).$$
This implies that for some $u\in \{v,w,x_1,x_2\}$,
$$|N_{H_{1,3}(n,n/3)}(u)\setminus N_H(u)|>n^3/(2^{12})>\alpha n^3,$$
a contradiction. \qed

\medskip

  Having proved the above two results, we are ready to complete the proof of Theorem~\ref{prefect-(1,3)} in the case when $H$ is close to
  some $H_{1,3}(n,n/3)$. For any $v\in V(H)$, let $N_H(v)$ denote the
  {\it link graph} of $v$ whose vertex set is $V(H)$ and edges set is $\{e\setminus \{v\}: e\in E(H)\}$.

\begin{lemma}\label{close}
  Let $n$ be a positive integer and $\varepsilon >0$ sufficiently small, and let $H$ be a balanced $(1,3)$-partite 4-graph with
   partition classes $Q,P$ and $3|Q|=|P|=n$.
 Suppose $H$ is  $\varepsilon$-close to some $H_{1,3}(n,n/3)$ with
 $P=V(H(n,n/3))$.
If $d_H(\{u,v\})\ge {n-1\choose 2}-{2n/3\choose 2}+1$ for all $u\in P$
and $v\in Q$, then $H$ has a perfect matching.
\end{lemma}

\pf Let $U,W$ denote the partition of $P=V(H(n,n/3))$ such that $|W|=|U|/2= n/3$.
Note that $|Q|=n/3$.
Let $B$ denote the set of $\sqrt{\varepsilon}$-bad vertices of $H$.
Since $H$ is $\varepsilon$-close to $H_{1,3}(n,n/3)$, we have $|B|\leq
4\sqrt{\varepsilon}n$. Let $Q\cap B=\{v_1,
\ldots, v_q\}$ and $Q=\{v_1, \ldots, v_{n/3}\}$, and
let $W'\subseteq W\setminus B$ such that $|W'|=n/3-(q+|W\cap B|)$.

First, we find a matching $M_0'$ in $H-W'$ covering $Q\cap B$. For this,  let
$F_i=N_{H}(v_i)- W'$ for $i\in [n/3]$. Note that, for $i\in [n/3]$,
$\delta_1(N_H(v_i))=\min \{d_H(\{u,v_i\}): u\in P\}\ge {n-1\choose
  2}-{2n/3\choose 2}+1$. Hence,
\begin{align*}
\delta_1(F_i)&\geq \delta_1(N_{H}(v_i))-\left({n-1\choose 2}-{n-|W'|-1\choose 2}\right)\\
&>{n-|W'|-1\choose 2} -{2n/3\choose 2}\\
& = {n-|W'|-1\choose 2}-{n-|W'|-(q+|W\cap B|) \choose 2}.
\end{align*}
Since $|B|\le 4\sqrt{\varepsilon}n$, $|W'|\ge
n/3-4\sqrt{\varepsilon}n$ and $q+|W\cap B|<(n-|W'|)/(2\cdot 3^4)$.  Hence
by Lemma~\ref{HZ-16}, $\{F_{1},\ldots,F_{q+|W\cap B|}\}$ admits a
rainbow matching, say $M_0$.
Let $M_{0}=\{e_i\in E(F_i) : i\in [q+|W\cap B|]\}$, and let
$M_0'=\{e_i\cup \{v_i\} :  i\in [q+|W\cap B|]\}$.
Then $M_0'$ is a matching in $H$ and $Q\cap B\subseteq V(M_0')$.

Next, we find a matching in $H_1:=H-V(M_0')$ covering $B\setminus
V(M_0')$, in two steps. Since $\varepsilon$ is very small, we can choose $\eta$ such that $\sqrt{\varepsilon} \ll
\eta\ll 1$. We divide $B\setminus V(M_0')$ to two disjoint sets $B_1,B_2$
such that, for each $x\in B\setminus V(M_0')$, $x\in B_1$ if, and only if,  $H_1$ has at
least $\eta n^3$ edges each of which contains $x$ and exactly one vertex in $W'$.

We greedily pick a matching $M_1$ in $H_1$ such that $B_{1}\subseteq
V(M_1)$ and every edge of $M_1$ contains at least one vertex from
$B_1$ and exactly one vertex from $W'$. This can be done since each time we pick an edge $e$ for a vertex $x\in B_1$, we have
at least $\eta n^3$ choices and at most $4(4\sqrt{\varepsilon}n)n^2$ ($\ll \eta n^3$ as $\sqrt{\varepsilon}\ll \eta$) of which intersect a previous chosen edge.

Now we find a matching $M_2$ in $H_2:=H_1-V(M_1)$ such that $B_2\subseteq V(M_2)$. Note that
\[
\delta_1(H_2)\geq \delta_1(H)-4|M_0'\cup M_1|n^2\ge \frac{n}{3}\left({n-1\choose 2}-{2n/3\choose 2}+1\right)-16\sqrt{\varepsilon}n^3.
 \]
Hence, for any $x\in B_2$, the number of edges containing $x$ and disjoint from $W'$ is at least
\[
\delta_1(H_2)-\eta n^3-|Q|{|W'|\choose 2}>\eta n^3,
\]
as $\sqrt{\varepsilon}\ll \eta\ll 1$ and $|Q|=|W'|=n/3$. Thus, since $\sqrt{\varepsilon}\ll \eta$,
we greedily find a matching $M_2$ in $H_1-V(M_1)$ such that $B_2\subseteq V(M_2)$, $M_2$ is disjoint from $W'$, and
every edge of $M_2$ contains at least one vertex from $B_2$.

Thus, $M_1\cup M_2$ gives the desired matching in $H_1:=H-V(M_0')$ covering $B\setminus V(M_0')$.
Note that $|M_0'\cup M_1\cup M_2|\le (q+|W\cap B|) +|B_1|+|B_2|  \leq
2|B|\le 8\sqrt{\varepsilon}n$. Also note that each vertex of $H-V(M_0'\cup M_1\cup
M_2)$ is $\sqrt{\varepsilon}$-good in $H$ (with respect to
$H_{1,3}(n,n/3)$). Thus, for every vertex $u\in U-V(M_0'\cup M_1\cup M_2)$, the number
of edges of $H-V(M_0'\cup M_1\cup M_2)$ containing $u$ and  exactly two vertices of $W-V(M_0'\cup M_1\cup M_2)$ is at least
\[
\frac{n}{3}{n/3\choose 2}-\sqrt{\varepsilon} n ^3 -4|M_0'\cup M_1\cup M_2|n^2>\eta n^3,
\]
as $\sqrt{\varepsilon}\eta\ll 1$. Hence,  we may greedily find a matching $M_2'$ in $H-V(M_0'\cup M_1\cup M_2)$
such that $|M_2'|=|M_2|$ and every edge  of $M_2'$ contains exactly two vertices of $W'$.

Let $M=M_0'\cup M_1\cup M_2\cup M_2'$ and $m=|M|$. Then $m\le 8\sqrt{\varepsilon}n$.  Let
$H_3=H-V(M)$. Let $H_{1,3}(n-3m,n/3-m)$ be obtained from $H_{1,3}(n,n/3)$ by removing $V(M)$. Then, for any $x\in V(H_3)$,
\begin{eqnarray*}
& &  |N_{H_{1,3}(n-3m,n/3-m)}(x)\setminus N_{H_3}(x)| \\
& \le & |N_{H_{1,3}(n,n/3)}(x)\setminus N_{H}(x)|+ |N_{H}(x)\setminus N_{H_3}(x)| \\
& \le & \sqrt{\varepsilon} n^3+4mn^2\\
&\le & \varepsilon^{1/5}n^3.
\end{eqnarray*}

Thus, every vertex of $H_3$ is $\varepsilon^{1/5}$-good with respect to $H_{1,3}(n-3m,n/3-m)$.
By Lemma \ref{good},  $H_3$ contains a perfect matching, say $M_3$.  Now $M_3\cup M$ is a perfect  matching in $H$.\qed

\section{Absorbing Lemma}

Our strategy to prove Theorem~\ref{prefect-(1,3)} is to find a small matching $M'$ in $H$ that can be used to ``absorb" small sets of
vertices,  find an almost perfect matching $M''$ in $H-V(M')$, and  then use $M'$ to absorb $V(H)\setminus V(M'\cup M'')$.
In this section, we prove such an  absorbing lemma for $(1,3)$-partite $4$-graphs. Our proof follows along the same lines
as in \cite{RRS06}.

\begin{lemma}\label{absorbing}
Let $n\in 3\Z$ be large enough and let $H$ be a $(1,3)$-partite $4$-graph with partition classes $Q,P$ such that $3|Q|=|P|$ and
$\delta_1(H)\ge (n/3)\left({n-1\choose 2}-{2n/3\choose 2}+1\right)$.
Let $\rho, \rho'$ be constants such that $0<\rho'\ll \rho\ll 1$.
Then $H$ has  a matching $M'$ such that  $|M'|\leq \rho n$ and,
for any subset $S\subseteq V(H)$
with $|S|\leq \rho' n$ and $3|S\cap Q|=|S\cap R|$,
$H[S\cup V(M)]$ has a perfect matching.
\end{lemma}
\pf
We call a balanced 12-element set $A\subseteq V(H)$ an {\it absorbing set} for a balanced $4$-element set $T\subseteq V(H)$
if $H[A]$ has a matching of size 3 and $H[A\cup T]$ has a matching of size $4$.
Denote by $\mathcal{L}(T)$ the collection of all absorbing sets for $T$. Then

\begin{itemize}
\item [(1)]  for every balanced  $T\in  {V(H)\choose 4}$, $|\mathcal{L}(T)|> 10^{-8}n^{12}/12!$.
\end{itemize}
Let  $T=\{u_0,u_1,u_2,u_3\}\in  {V(H)\choose 4}$ be balanced, with
$u_0\in Q$ and $u_1,u_2,u_3\in P$. We form an absorbing set for $T$ by choosing four pairwise disjoint 3-sets $U_0,U_1,U_2,U_3$ in order.

First, we choose a 3-set $U_0\subseteq P\setminus T$ such that $U_0\cup \{u_0\}\in E(H)$. The number of choices
for $U_0$ is  at least
$$d_H(u_0)-3{n-3\choose 2}> \delta_1(H)-3{n-1\choose 2}>\frac{n}{9}{n-1\choose 2}.$$

Now fix a choice of $U_0$, and let $U_0=\{w_1,w_2,w_3\}$. Note that, for each $i\in [3]$,
$N_H(u_i)\cup N_H(w_i)$ is a subset of the union of
$\{\{x_0,x_1,x_2\}: x_0\in Q, x_1,x_2\in P\}$.
Hence, $|N_H(u_i)\cup N_H(w_i)|\le  \frac{n}{3}{n\choose 2}$. Thus,
for $i\in [3]$,
$$
|N_{H}(u_i)\cap N_{H}(w_i)|\geq \frac{2n}{3}\left({n-1\choose 2}-{2n/3\choose 2}+1\right)-\frac{n}{3}{n\choose 2}\\
\geq  \frac{n}{30}{n-1\choose 2}.$$

For $i\in [3]$, we choose 3-sets $U_i$ from $V(H)\setminus T\setminus \bigcup_{j=0}^{i-1}U_j$ such that $U_i\cup \{u_i\}$ and $U_i\cup \{w_i\}$
are both edges of $H$. For each choice of $U_j$ with $0\le j\le i-1$, the number of choices for $U_i$ is at least
$$ |N_{H}(u_i)\cap N_{H}(w_i)|-13(n/3)n \ge \frac{n}{30}{n-1\choose 2}-13n^2/3>  \frac{n}{50}{n-1\choose 2}.$$

Let $A=\bigcup_{i=0}^3 U_i$. Then
$\{U_i\cup \{w_i\}: i\in [3]\}$ is a matching in $H[A]$, and $\{U_i\cup \{u_i\}: i\in [3]\cup \{0\}\}$  is a matching in $H[A\cup T]$.
Thus $A$ is an absorbing set for $T$.
Since there are more than $10^{-8}n^{12}$ choices of $(U_0,U_1,U_2,U_3)$,
there are more than $10^{-8}n^{12}/12!$ absorbing sets for $T$.  $\Box$

\medskip

Now, form a family $\F$ of subsets of $V(H)$ by selecting each of the ${n/3\choose 3}{n\choose 9}$
possible balanced 12-sets independently with probability
\[
p=\frac{\rho n}{2{{n/3\choose 3}{n\choose 9}}}.
\]
Then, it follows from  Chernoff's bound that, with probability $1-o(1)$ (as $n\rightarrow \infty$),
\begin{itemize}
\item [(3)] $|\F|\leq \rho n$, and
\item [(4)] $|\mathcal{L}(T)\cap F|\geq p |\mathcal{L}(T)|/2\geq 10^{-10}\rho n$ for all balanced $T\in {V(H)\choose 12}$.
\end{itemize}

Furthermore, the expected number of intersecting pairs of sets in $\F$ is at most
\[
{n/3\choose 3}{n\choose 9} \left[3{n/3-1\choose 2}{n\choose 9}+9{n-1\choose 8}{n/3\choose 3} \right]p^2< \rho^{1.5}n.
\]
Thus, using Markov's inequality, we derive that, with probability at least 1/2,
\begin{itemize}
\item [(5)] $\F$ contains at most $2\rho^{1.5}n$ intersecting pairs.
\end{itemize}

Hence, with positive probability, $\F$ satisfies (3), (4), and (5).
Let $\F'$ be obtained from $\F$ by removing one set from each intersecting pair and deleting all non-absorbing sets.
Then  $F'$ consists of pairwise disjoint absorbing sets, such that for each $T\in {V(H)\choose 4}$,
\[
|\mathcal{L}(T)\cap \F'|\geq 10^{-10}\rho n/2.
\]

 Since $F'$ consists only of pairwise disjoint absorbing sets, $H[V(F')]$ has a perfect matching, say  $M'$. Then $|M'| \le \rho n$.
 To complete the proof, take an arbitrary $S\subseteq V(H)\setminus V(M)$ with $|S|\leq \rho' n$ and $3|S\cap Q|=|S\cap P|$,
 where $\rho'\le 10^{-10}\rho/2$.
Note that $S$ can be partitioned into $t$ balanced $4$-sets, say $T_1,\ldots, T_t$, for some $t\le \rho' n/4<10^{-10}\rho n/2$.
We can greedily
choose distinct absorbing sets $A_i\in \F'$ in order for $i=1, \ldots, t$, such that $H[A_i\cup T_i]$ has a perfect matching.
Hence, $H[S\cup V(M')]$ has a perfect matching as required. \qed

\section{Perfect fractional matching}


When $H$ is not close to any $H_{1,3}(n,n/3)$  we will show that $H$ contains a $(1,3)$-partite 4-graph $H'$ in which no independent set
is too large (see Lemma~\ref{indep}) and we then use this property of
$H'$ to show that $H'$ has a perfect fractional  matching (see Lemma~\ref{pfrac}).

To obtain $H'$, we use the hypergraph container method developed by
Balogh, Morris and Samotij \cite{BMS15} and, independently, by Saxton and Thomason \cite{ST15}.
A family ${\cal F}$ of subsets    of a set $V$  is
said to be {\it increasing} if, for any $A\in \F$ and
$B\subseteq V$,   $A\subseteq B$ implies $B\in \F$.
Let $H$ be a hypergraph. We use $v(H), e(H)$ to denote the number of
vertices,  number of edges in $H$, respectively. We also use $\Delta_l(H)$ to denote the maximum $l$-degree of $H$, and $\I(H)$ to
denote the collection of all independent sets in $H$. Let $\varepsilon>0$
and  let ${\cal F}$ be a family of subsets of $V(H)$. We say that $H$ is \textit{$(\mathcal F , \varepsilon)$-dense} if $e(H[A])\ge \varepsilon e(H)$ for every $A\in \F$. We use $\overline{\F}$ to denote the family consisting of subsets of $V(H)$ not in $\F$.

\begin{lemma}[Balogh, Morris, and Samotij, 2015]\label{thm2.2}
        For every $k \in \N$ and all positive $c$ and $\varepsilon$, there exists a positive constant $C$ such that the following holds.
        Let $H$ be a $k$-graph and let $\F$ be an increasing family of subsets of $V(H)$ such that
        $|A| \ge \varepsilon v(H)$ for all $A \in \F$.
        Suppose that $H$ is $(\mathcal F , \varepsilon)$-dense and $p \in (0, 1)$ is such that, for every $l \in [k]$,
        \[\Delta_l(H)\le cp^{l-1}\frac{e(H)}{v(H)}.
        \]
        Then there exist a family $\S\subseteq \binom{V(H)}{\le Cp v(H)}$ and functions $f: \S \to \overline{\F}$ and $g: \I (H) \to \S$ such that, for every $I\in \I (H)$,
        \[g(I)\subseteq I \quad and \quad I\setminus g(I)\subseteq f(g(I)).\]
\end{lemma}

In order to apply Lemma~\ref{thm2.2} we need a family $\F$ of subsets
of $V(H)$ so that $H$ is  $(\F, \varepsilon)$-dense,
which is possible when $H$ is not close to any $H_{1,3}(n,n/3)$.

\begin{lemma}\label{dense}
Let $\rho, \varepsilon$ be reals such that $0<\rho\le \varepsilon/4\ll
1$, let $n\in 3\Z$ be large, and
let $H$ be a $(1,3)$-partite $4$-graph with partition classes $Q,P$
such that $3|Q|=|P|=n$ and
$d_H(\{u,v\})\geq {n-1\choose 2}-{2n/3\choose 2}-\rho n^2$ for any
$v\in Q$ and $u\in P$. If $H$ is not $\varepsilon$-close to any $H_{1,3}(n,n/3)$, then $H$
is $(\F, \varepsilon/6)$-dense, where $\F=\{A\subseteq V(H) : |A\cap Q|\ge (1/3-\varepsilon/8) n \mbox{ and } |A\cap P|\ge (2/3-\varepsilon/8) n\}$.
\end{lemma}

\begin{proof}
	Suppose to the contrary that there exists $A\subseteq V(H)$ such
       that $|A\cap Q|\ge (1/3-\varepsilon/8) n$, $|A\cap P|\ge (2/3-\varepsilon/8) n$, and $e(H[A])\le \varepsilon e(H)/6$.
       Choose such $A$ that  $|P\setminus A|\ge n/3$ and let $W\subseteq P\setminus A$ such that
        $|W|=n/3$.
        Let $A_1=A\cap P$ and $A_2=A\cap Q$, and let $B_1=P\setminus W\setminus A_1$, $B_2=Q\setminus A_2$, and $B=B_1 \cup B_2$. Then
        $|A_1|\le 2n/3$ and, by the choice of $A$, $|B_1|\le \varepsilon n/8$ and $|B_2|\le \varepsilon n/8$.

        Let $U=P\setminus W=A_1\cup B_1$ and let $H_0$ denote the $H_{1,3}(n,n/3)$ with partition classes $Q, U,W$.
        We derive a contradiction by showing
        that $|E(H_0)\setminus E(H)|<\varepsilon n^4$.
        Note that each $f\in E(H_0)\setminus E(H)$ intersects $U$. So
        \[
        	|E(H_0)\setminus E(H)|\le |\{f\in E(H_0): f\cap B_1\ne \emptyset \}|+|\{f\in E(H_0)\setminus E(H): f\cap A_1\ne\emptyset\}|.
	\]
	
        Since $|B_1|\le \varepsilon n/8$, we have $|\{f\in E(H_0): f\cap B_1\ne \emptyset \}|\le |B_1||Q||P|^2/2\le \varepsilon n^4/48$.
	To bound $|\{f\in E(H_0)\setminus E(H): f\cap A_1\ne\emptyset\}|$, we note that, for each fixed $u\in A_1$,
       \[
       	|\{f\in E(H): u\in f,\ f\cap B\ne \emptyset\}|\le |B_1||P||Q|+|B_2||P|^2/2< \varepsilon n^3/8,
	\]
       and that ,  for each $f\in E(H)$ with $u\in f$, we have  $f\cap B\ne
        \emptyset$, or $f\subseteq A$, or $f\in E(H_0)$. So for any $u\in A_1$,
          \begin{align*}
             & |\{f\in E(H) : u\in f,\ f\in E(H_0)|\\
             \ge & d_H(u)-|\{f\in E(H) : u\in f,\ f\cap B\ne \emptyset\}| -|\{f\in E(H) : u\in f,\ f\subseteq A\}|\\
              \ge & d_H(u) -\varepsilon n^3/8-d_{H[A]}(u).
          \end{align*}
	Hence,
	\begin{align*}
		 & |\{f\in E(H_0)\setminus E(H): f\cap A_1\ne\emptyset\}|\\
		\le & \sum_{u\in A_1}|\{f\in E(H_0)\setminus E(H): u\in f\}|\\
		 \le &\sum_{u\in A_1}\left( d_{H_0}(u)- |\{f\in E(H) : u\in f,\ f\in E(H_0)|\right) \\
		\le & \sum_{u\in A_1}\left(d_{H_0}(u)- d_H(u)+\varepsilon n^3/8+ d_{H[A]}(u)\right).
	\end{align*}
	Since for $u\in A_1$, $d_{H_0}(u)=\left({n-1\choose 2}-{2n/3-1\choose 2}\right)n/3$
	and $d_H(u)=\sum_{v\in Q}d_H(\{u,v\})\ge \left({n-1\choose 2}-{2n/3\choose 2}-\rho n^2\right)n/3$, we have $d_{H_0}(u)-d_H(u)\le\rho n^3/3$ (for large $n$).
	Hence,
        \begin{align*}
		|E(H_0)\setminus E(H)|
		&\le \varepsilon n^4/48 + |A_1| \left(\rho/3+3\varepsilon/8 \right)n^3+\sum_{u\in A_1} d_{H[A]}(u)\\
		&\le \left(\varepsilon/48 + 4\rho/9 +
                  \varepsilon/4\right) n^4 + 3e(H[A]) \quad (\mbox{ since $|A_1|\le 2n/3$})\\
		&\le  \left(1/48+1/9+1/4\right) \varepsilon n^4 +
                  3\varepsilon n^4/6 \quad (\mbox{ since $e(H[A])\le
                  \varepsilon e(H)/6$})\\
		&< \varepsilon n^4,
	\end{align*}
	a contradiction.
\end{proof}


We now use Lemma~\ref{thm2.2} to control the independence number of a random subgraph.

\begin{lemma}\label{indep}
        Let $c, \varepsilon', \alpha_1,\alpha_2$ be positive reals, let $\gamma>0$ with $\gamma \ll \min\{\alpha_1,\alpha_2\}$, let $k,n$ be positive integers,
        and let $H$ be a $(1,3)$-partite $4$-graph with partition classes $Q,P$ such that $3|Q|=|P|=n$,   $e(H)\ge cn^4$,  and $e(H[F])\ge
        \varepsilon' e(H)$ for all $F\subseteq
        V(H)$ with $|F\cap P|\ge \alpha_1 n$ and $|F\cap Q|\ge \alpha_2 n$.
         Let $R\subseteq V(H)$ be obtained  by taking each vertex of
           $H$ uniformly at random with probability $n^{-0.9}$.
        Then, with probability at least $1-n^{O(1)}e^{-\Omega (n^{0.1})}$, every independent set $J$ in $H[R]$
             satisfies $|J\cap P|\le (\alpha_1 +\gamma+o(1))n^{0.1}$ or $|J\cap Q|\le (\alpha_2 +\gamma+o(1))n^{0.1}$.
\end{lemma}

   \pf
        Define $\F :=\left\{A\subseteq V(H) \ : \  e(H[A])\ge \varepsilon'
          e(H) \mbox{ and } |A|\ge \varepsilon' n\right\}$.
        Then $\F$ is an increasing family,  and $H$ is $(\F, \varepsilon')$-dense.
        Let $p=n^{-1}$ and $v(H)=4n/3$. Then, for $l\in [4]$,
        \[\Delta_l(H)\le \binom{4n/3}{4-l}\le (4n/3)^{4-l}\le (4/3)^{4-l}c^{-1}n^{-l}e(H)=(4/3)^{4-l+1}c^{-1} p^{l-1}\frac{e(H)}{v(H)}.
        \]
        Thus by Lemma~\ref{thm2.2}, there exist constant $C$,
        family $\S\subseteq \binom{V(H)}{\le C}$, and function $f: \S\to \overline{\F}$,
        such that every independent set in $H$ is contained in some $T\in \T:= \left\{F\cup S :  F\in f(\S) , S\in \S \right\}$.
        Since $\S\subseteq \binom{V(H)}{\le C}$, $|\S|\le C (4n/3)^C$
        and, hence, $$|\T|=|\S| |f(\S)|\le |\S|^2\le C^2(4n/3)^{2C}.$$

        Since   for $T\in \T$ it is possible that $|T\cap P|< \alpha_1 n+C$ or $|T\cap Q|< \alpha_2 n +C$,  we need to make the sets in
        $\T$ slightly larger in order to apply Chernoff's inequality.
        For each $T\in \T$, let $T'$ be a set obtained from $T$ by adding vertices such that
        $|T'\cap P|=\max\{|T\cap P|, \lceil \alpha_1 n+C \rceil\}$ and $|T'\cap Q|=\max\{|T\cap Q|,   \lceil \alpha_2 n +C \rceil\}$.
         Let $\T':=\{T': T\in \T\}$.  Then      $$|\T'|\le |\T|\le  C^2(4n/3)^{2C}.$$

         Note that for each fixed $T'\in \T'$, we have $|R\cap T'\cap P|\sim Bi\left(|T'\cap P|, n^{-0.9} \right)$ and
        $|R\cap T'\cap Q|\sim Bi\left(|T'\cap Q|, n^{-0.9} \right)$. Hence, $\E (|R\cap T'\cap P|)=n^{-0.9}|T'\cap P|$ and
        $\E (|R\cap T'\cap Q|)=n^{-0.9}|T'\cap Q|$.
        Applying Chernoff's bound to $|R\cap T'\cap P|$ and $|R\cap T'\cap Q|$ by taking
        $\lambda= \gamma n^{0.1}$, we have,
        \begin{align*}
        	\P\left(\big| |R\cap T'\cap P| - n^{-0.9}|T'\cap P|  \big| \ge \lambda  \right) \le e^{-\Omega(\lambda^2/ (n^{-0.9}|T'\cap P|)}
	&\le e^{-\Omega(n^{0.1})}, \mbox{ and } \\
        	 \quad \P\left(\big| |R\cap T'\cap Q| - n^{-0.9}|T'\cap Q|  \big| \ge \lambda  \right) \le e^{-\Omega(\lambda^2/ (n^{-0.9}|T'\cap Q|)}
	&\le  e^{-\Omega(n^{0.1})}.
	\end{align*}
        So with probability at most $2e^{-\Omega(n^{0.1})}$, $|R\cap
        T'\cap P|\ge n^{-0.9}|T'\cap P|+\lambda\ge  (\alpha_1 +\gamma + C/n)n^{0.1}$ and
        $|R\cap T'\cap Q|\ge n^{-0.9}|T'\cap Q|+\lambda\ge (\alpha_2 +\gamma + C/n)n^{0.1}$.

        Therefore,
        with
        probability at most $2C^2n^{2C}e^{-\Omega(n^{0.1})}$, there exists some $T'\in \T'$ such that
       $|R\cap T'\cap P|\ge (\alpha_1 +\gamma + C/n)n^{0.1}$ and $|R\cap T'\cap Q|\ge (\alpha_2 +\gamma + C/n)n^{0.1}$.
        Hence, with probability at least $1-2C^2n^{2C}e^{-\Omega(n^{0.1})}$,
        $|R\cap T'\cap P|< (\alpha_1 +\gamma + C/n)n^{0.1}$ or $|R\cap T'\cap Q|< (\alpha_2 +\gamma + C/n)n^{0.1}$ for all $T'\in \T'$.

        Now let $J$ be an  independent set in $H[R]$. Then $J$ is also an independent set in $H$; so there
        exist  $T\in \T$ and $T'\in \T'$ such that $J\subseteq T\subseteq T'$.
        Thus $J\subseteq R\cap T'$; so $|J\cap P|\le |R\cap T'\cap P|$ and $|J\cap Q|\le |R\cap T'\cap Q|$.  Hence,
        with probability at least $1-2C^2n^{2C} e^{-\Omega
          (n^{0.1})}$,  $|J\cap P|\le (\alpha_1 +\gamma+C/n)n^{0.1}$
      or $|J\cap Q|\le (\alpha_2 +\gamma+C/n)n^{0.1}$. \qed

\medskip

To show that a $(1,3)$-partite 4-graph with no large independent set
has a perfect fractional matching,
we need a result from \cite{LYY} about stable 2-graphs. A 2-graph $G$ is {\it stable}
with respect to a labeling $u_1,\ldots, u_n$ of its vertices  if, for
any $i,j,k,l\in [n]$ with $k\le i$ and $l\le j$, $u_iu_j\in E(G)$ implies $u_ku_l\in E(G)$.

\begin{lemma}\label{3-graph-frac}
Let $c, \rho$ be constants such that $0<\rho\ll 1$ and $0<c<1/2$, let $m, n$ be positive integers such that
$n$ is sufficiently large and $cn\leq m\leq n/2-1$, and  let $G$ be a
2-graph with   $\nu(G)\leq m$.
Suppose $G$ is stable with respect the ordering of its vertices $u_1, \ldots, u_n$.
If $e(G)> {n\choose  2}-{n-m\choose 2}-\rho n^2$, then
$G$ is $2\sqrt{\rho}$-close to the graph with vertex $V(G)$ and edge
set $\{e\in {V(G)\choose 2}: e\cap \{u_i: i\in [n/3-1]\}\ne \emptyset$.

\end{lemma}

\medskip

We now prove the main result of this section.
A {\it fractional matching} in a $k$-graph
  $H$ is a function $w: E\rightarrow [0,1]$ such that for any
  $v\in V(H)$, $\sum_{\{e\in E: v\in e\}}w(e)\le 1$. A fractional matching is  {\it perfect} if $\sum_{e\in E}w(e)=|V(H)|/k$.

\begin{lemma}\label{pfrac}
Let $\rho, \varepsilon$ be constants with  $0<\varepsilon\ll 1$ and $0< \rho <\varepsilon^{12}$, and
let $H$ be a $(1,3)$-partite $4$-graph with partition classes $Q,P$ such that  $3|Q|=|P|=n$.
Suppose
$d_H(\{u,v\})> {n-1\choose 2}-{2n/3\choose 2}-\rho n^2$ for any $v\in Q$ and $u\in P$. If
$H$ contains no independent set $S$ with
$|S\cap Q|\geq n/3-\varepsilon^2 n$ and $|S\cap P|\geq
2n/3-\varepsilon^2 n$, then $H$ contains a perfect fractional matching.

\end{lemma}

\begin{proof}
Let $\omega:V(H)\rightarrow \mathbb{R}^+\cup \{0\}$ be a minimum
fractional vertex cover of $H$, i.e., $\sum_{x\in e}\omega(x)\ge 1$
for $e\in E(H)$ and, subject to this, $\sum_{x\in V(H)}\omega(x)$ is
minimum.
Let $P=\{u_1,\ldots,u_n\}$ and $Q=\{v_1, \ldots, v_{n/3}\}$, such that $\omega(v_1)\geq \cdots \geq \omega(v_{n/3})$ and $\omega(u_1)\geq \cdots \geq \omega(u_n)$.
Let $H'$ be the $(1,3)$-partite $4$-graph with vertex set $V(H)$ and edge set $E(H')=E'$, where
\[
E'=\left\{e\in {V(H)\choose 4}\ : \ |e\cap Q|=1 \ \mbox{and }\sum_{x\in e}\omega(x)\geq 1\right\}.
\]

Note that $\omega$ is also a minimum fractional vertex cover of $H'$. So $\omega(H)=\omega(H')$, where $\omega(H):=\sum_{v\in V(H)}\omega(v)$ and
$\omega(H'):=\sum_{v\in V(H')}\omega(v)$.
Let $\nu_f(H)$ and $\nu_f(H')$ denote the maximum fractional matching numbers of $H$ and $H'$, respectively;
then by the Strong Duality Theorem of linear programming, $\nu_f(H)=\omega(H)$ and  $\nu_f(H')=\omega(H')$.
Thus $\nu_f(H)=\nu_f(H')$ and, hence, it suffices to show that $H'$ has a perfect matching.

Next, we observe that the edges of $H'$ form a stable family with respect to the above ordering of vertices in $P$ and $Q$:
for any  $e_1=\{v_{i_1},u_{i_2},u_{i_3},u_{i_4}\}$ and $e_2=\{v_{j_1},u_{j_2},u_{j_3},u_{j_4}\}$ with
$i_l\geq j_l$ for $1\leq l\leq 4$, $e_2\in E(H')$ implies  $e_1\in E(H')$.
To see this, note that, since $i_l\geq j_l$ for $1\leq l\leq 4$,
we have $\omega(v_{i_1})\geq \omega(v_{j_1})$ and $\omega(u_{i_l})\geq \omega(u_{j_l})$ for $2\leq l\leq 4$.
If $e_2\in E(H')$ then $ \sum_{x\in e_2}\omega(x)\geq 1$; so $\sum_{x\in e_1}\omega(x)\geq 1$ and, hence,  $e_1\in E(H')$.

Let $G$ denote the graph with vertex set $P$ and edge set formed by
$N_{H'}(\{v_{n/3},u_n\})$. Then $G$ is stable with respect to $u_1,\ldots,u_n$. Note that $e(G)>{n-1\choose 2}-{2n/3\choose 2}-\rho n^2$ (by assumption).
Since the edges of $H'$ form a stable family,  $\{u,v\}\cup e\in E(H')$
for all $u\in P, v\in Q$, and  $e\in E(G)$.
Thus, if $G$ contains a matching  $M:=\{e_1, \ldots, e_{n/3}\}$ then
let $x_1,\ldots, x_{n/3}\in P\setminus V(M)$; we see that
  $\{\{v_i, x_i\}\cup e_i\in E(H'): i\in [n/3]\}$ is a perfect matching in $H'$.

Thus, we may assume $\nu(G)< n/3$.
Hence, by Lemma~\ref{3-graph-frac}, $G$ is $2\sqrt{\rho}$-close to the graph with vertex $V(G)$ and edge
set $\{e\in {V(G)\choose 2}: e\cap \{u_i: i\in [n/3-1]\}\ne \emptyset$. Recall that $e(G)>{n-1\choose 2}-{2n/3\choose 2}-\rho n^2$.
Therefore, $G$ has at most $\sqrt{2\sqrt{\rho}}n$ vertices in $\{u_j\ |\ j\in [n/3-1]\}$ of degree less than $n-1-\sqrt{2\sqrt{\rho}}n$.
Since $G$ is stable with respect to $u_1,\ldots, u_n$, we have
$d_G(u_{n/3-\sqrt{2\sqrt{\rho}}n})\ge n-1-\sqrt{2\sqrt{\rho}}n$.

Since $\rho<\varepsilon^{12}$ and $H$ contains no independent set $S$  such that
$|S\cap Q|\geq n/3-\varepsilon^2 n$ and $|S\cap P|\geq 2n/3-\varepsilon^2 n$, we may form a matching $M_0$ of size $\sqrt{2\sqrt{\rho}}n$ in
$H-\{u_1,\ldots, u_{n/3}\}$ by greedily choosing edges.

Since $d_G(u_{n/3-\sqrt{2\sqrt{\rho}}n})\ge n-1-\sqrt{2\sqrt{\rho}}n$, $G-V(M_0)$ has a matching $M$ of size $n/3-\sqrt{2\sqrt{\rho}}n$ which can be found by
greedily choosing distinct neighbors of $u_i$,  $1\le i\le
n/3-\sqrt{2\sqrt{\rho}}n$, in $V(G)\setminus V(M_0)$.
 Since  $\{u,v\}\cup e\in E(H')$ for $u\in P, v\in Q$, and $e\in M$,
we may extend $M$ to a matching $M'$ of size $|M|$ in $H'-M_0$.
Then $M'\cup M_0$ gives a perfect matching in $H'$.
\end{proof}

\section{Almost perfect matching}

In this section, we use Lemmas~\ref{pfrac} and ~\ref{indsub} to find a ``near regular" spanning subgraph of $H$.
The discussion here follows that in \cite{AFH12}.
We need to find a sequence of random subgraphs of a balanced (1,3)-partite 4-graph
 and use them to find a subgraph on which a ``R\"odl nibble" result can be applied.

 First, we show how to find such a sequence.
 The following result is a lemma in \cite{LYY}, which was essentially the first of the two round randomization in \cite{AFH12}.

 \begin{lemma}\label{lem1-5}
        Let $n>k>d>0$ be integers with $k\ge 3$
       and let $H$ be a $k$-graph on $n$ vertices. {\color{blue}Let $0<c<1$ be a constant and let $J\subseteq V(H)$ such that $|J|=cn$.}
	Take $n^{1.1}$ independent copies of $R$ and denote them by $R^i$, $1\le i\le n^{1.1}$, where $R$ is chosen from $V(H)$ by taking each vertex uniformly at random with probability $n^{-0.9}$ and then deleting less than $k$ vertices  uniformly at random so that $|R|\in k\Z$.
        For each $X\subseteq V(H)$, let $Y_X:=|\{i: \ X\subseteq R^i\}|$ and $\bd^i_X:=|\{e\setminus X: X\subseteq e \mbox{ and } e\setminus X\subseteq R^i\}|$. Then, with probability at least $1-o(1)$, we have
	\begin{itemize}
              \setlength{\itemsep}{0pt}
		\setlength{\parsep}{0pt}
		\setlength{\parskip}{0pt}
            \item [$(i)$]  $Y_{\{v\}}=(1+o(1)) n^{0.2}$ for $v\in V(H)$,
            \item [$(ii)$] $Y_{\{u,v\}}\le 2$ for distinct $u, v \in V(H)$,
            \item [$(iii)$] $Y_e\le 1$ for  $e \in E(H)$,
            \item [$(iv)$] $|R^i| =(1+o(1))n^{0.1}$ for $i=1,\ldots, n^{1.1}$, and
            \item [$(v)$] if  $\mu, \rho'$ are constants with $0<\mu\ll \rho'$, $n/k-\mu n\le m\le n/k$, and
             $\delta_d(H)\ge {n-d\choose
                   k-d}-{n-d-m\choose k-d}-\rho' n^{k-l}$,  then
                for any positive real $\xi\ge 2\rho'$, we have
               $$\bd^{i}_D >  {|R^i|-d\choose k-d}-{|R^i|-d-|R^i|/k \choose k-d}-\xi |R^i|^{k-d}$$
      for all $i= 1, \dots ,n ^{1.1}$ and all $D\in { V(H) \choose d}$,
      \item[$(vi)$] {\color{blue}$|R_i\cap J|\sim cn^{0.1}$ for $i=1,\ldots, n^{1.1}$.}
	\end{itemize}
\end{lemma}

Since we work with balanced (1,3)-partite 4-graphs, we need to make sure each random subgraph taken is also balanced.
 So we slightky modify the randomization process in the above lemma. We first fix an arbitrary small set $S\subseteq V(H)$.
 Each time we obtain a random copy $R$, we delete some vertices in $R\cap S$ so that the resulting graph
 is balanced. We can do so in a way that, with high probability, all properties in Lemma~\ref{lem1-5} remain (approximately) true.

\begin{lemma}\label{indsub}
	Let $n$ be a sufficiently large positive integer, and let $H$ be a $(1,3)$-partite $4$-graph with partition classes
        $Q,P$ such that $3|Q|=|P|=n$.
	Let $S\subseteq V(H)$ be a set of vertices such that $|S\cap Q|=n^{0.99}/3$ and $|S\cap P|=n^{0.99}$.
	Take $n^{1.1}$ independent copies of $R_+$ and denote them by $R_+^i$, $1\le i\le n^{1.1}$,
        where $R_+$ is chosen from $V(H)$ by taking each vertex
        uniformly at random with probability $n^{-0.9}$. Define
        $R_-^{i}=R_+^i\setminus S$ for $1\le i\le n^{1.1}$.
	
	Then with probability $1-o(1)$, for any sequence $R^i$,
        $1\le i\le n^{1.1}$,  satisfying $R_-^i\subseteq R^i\subseteq R_+^i$, all of the following hold:
	\begin{itemize}
		\setlength{\itemsep}{0pt}
		\setlength{\parsep}{0pt}
		\setlength{\parskip}{0pt}
		\item[ $(i)$] $|R^i|=(4/3+o(1))n^{0.1}$ for all $i=1,\dots, n^{1.1}$.
		\item[ $(ii)$] For each $X\subseteq V(H)$, let
                  $Y_X:=|\{i: X\subseteq R^i\}|$, then,
			\begin{itemize}
				\setlength{\itemsep}{0pt}
				\setlength{\parsep}{0pt}
				\setlength{\parskip}{0pt}
				\item[ $(iia)$] $Y_{\{v\}}\le (1+o(1))n^{0.2}$ for $v\in V(H)$,
				\item[ $(iib)$] $Y_{\{v\}}= (1+o(1))n^{0.2}$ for $v\in V(H)\setminus S$,
				\item[ $(iic)$] $Y_{\{u,v\}}\le 2$ for distinct $u,v\in V(H)$, and
				\item[ $(iid)$] $Y_e\le 1$ for $e\in E(H)$.
			\end{itemize}
		\item[ $(iii)$] For each $X\in {V(H)\choose 2}$, let $\bd_X^i=|N_H(X)\cap \binom{R^i}{2}|$. If $\rho>0$
      is a constant and $d_H(\{u,v\})\geq {n-1\choose 2}-{2n/3\choose
        2}-\rho n^2$ for all $v\in Q$ and $u\in P$, then
       for any constant $\xi\ge 5\rho$, we have
			\[
				\bd_{\{u,v\}}^i>\binom{|R^i\cap P|-1}{2}-\binom{2|R^i\cap P|/3}{2}-\xi |R^i\cap P|^2.
			\]
      for all $i=1,\dots,n^{1.1}$, $v\in Q$, and $u\in P$.
	\end{itemize}
\end{lemma}

\begin{proof}
	Note that $\E(|R_+^i|)= (4n/3)\cdot n^{-0.9}=4n^{0.1}/3$, and
$$\E(|R_-^i|)= (4n/3-4n^{0.99}/3)\cdot n^{-0.9}=4n^{0.1}/3-4n^{0.09}/3.$$ By Chernoff's inequality,
	\[
		\P(|R_+^i|-4n^{0.1}/3\ge n^{0.095})\le
                e^{-\Omega(n^{0.09})}
        \]
 and
        \[
                \P(|R_-^i|-(4n^{0.1}/3-4n^{0.09}/3)\le -n^{0.095})\le e^{-\Omega(n^{0.09})}.
	\]
In particular,  (i) holds with probability at least $1-e^{-\Omega(n^{0.09})}$.
	
	Let  $Y^+_X:=|\{i: X\subseteq  R_+^i\}|$
	for  $X\subseteq V(H)$. Then $Y^+_X\sim Bi(n^{1.1},
        n^{-0.9|X|})$ and $Y_X\le Y^+_X$ for all $X\subseteq V(H)$, and $Y_X= Y^+_X$ for all $X\subseteq V(H)\setminus S$.
       	Then by Lemma~\ref{lem1-5}, (iic) and (iid) hold with probability $1-o(1)$.
	
       For each $v\in V(H)$, $\E(Y^+_{\{v\}})=n^{0.2}$, thus by Chernoff's inequality,
	\[
		\P\left(\left|Y^+_{\{v\}}-n^{0.2}\right|\ge n^{0.15}\right)\le e^{-\Omega(n^{0.1})}.
	\]
	Thus (iia) and (iib) hold with probability at least $1-e^{-\Omega(n^{0.1})}$.

	Let $\sd_X^i=\left|N_H(X)\cap \binom{R_-^i}{2}\right|$. To
        prove $(iii)$,  since $n$
        is sufficiently large, it suffices to show that for all $v\in Q$ and $u\in P$,
	\[
		\sd_{\{u,v\}}^i>\binom{n^{0.1}-1}{2}-\binom{2n^{0.1}/3}{2}-\xi n^{0.2}/2.
	\]
	Conditioning on $|R_+^i|<4n^{0.1}/3 - n^{0.095}$ and $|R_-^i|
        > (4n^{0.1}/3-4n^{0.01}/3) -n^{0.095}$ for all $i$, we have,
        for all $v\in Q$ and $u\in P$,
	\begin{align*}
		\E(\sd_{\{u,v\}}^i) & = d_{H-S}(\{u,v\})(n^{-0.9})^2 \\
			& \ge (1-o(1)) \left({n-1\choose 2}-{2n/3\choose 2}-\rho n^2\right) (n^{-0.9})^2 \\
			& \ge \binom{n^{0.1}-1}{2}-\binom{2n^{0.1}/3}{2}-2\rho n^{0.2},
	\end{align*}
	where the first inequality holds because $|S|=4n^{0.99}/3$ (and,
        hence, $d_{H-S}(\{u,v\})=(1-o(1))d_H(\{u,v\})$). In particular, $\E\left(\sd_{\{u,v\}}^i\right)=\Omega(n^{0.2})$.
	Next, we apply Janson's Inequality (Theorem 8.7.2 in \cite{AS}) to bound the deviation of $\sd^i_{\{u,v\}}$.
         Write $\sd^i_{\{u,v\}}=\sum_{e\in N_H({\{u,v\}})} X_e$,
         where $X_e=1$ if $e\subseteq R_-^i$ and $X_e=0$ otherwise. Then
         \[\Delta := \sum_{e\cap f \ne \emptyset} \P(X_{e}=X_{f}=1)\le \binom{n-1}{2}\binom{2}{1}\binom{n-3}{1}(n^{-0.9})^3
         \]
         and, thus, $ \Delta = O(n^{0.3})$.
         By Janson's inequality, for any constant $\gamma>0$,
         \[\P\left(\sd^i_{\{u,v\}}\le(1-\gamma)\E(\sd^i_{\{u,v\}})\right)\le e^{-\gamma^2 \E(\sd^i_{\{u,v\}})/(2+\Delta/\E(\sd^i_{\{u,v\}}))}=e^{-\Omega(n^{0.1})}.
         \]
         Since $\xi\ge 5\rho$ (and taking $\gamma$ sufficiently small), the union bound implies that, with probability
         at least $1-n^{2+1.1}e^{-\Omega(n^{0.1})}$, for all $v\in Q$
         and $u\in P$ and for all $i\in [n^{1.1}]$,
	\[
		\sd_{\{u,v\}}^i>\binom{n^{0.1}-1}{2}-\binom{2n^{0.1}/3}{2}-\xi n^{0.2}/2.
	\]

         Thus,  $(iii)$ holds with probability at least
	\[
         	(1-n^{1.1}e^{-\Omega(n^{0.09})}) (1- n^{2+1.1}e^{-\Omega(n^{0.1})})>1- n^{4}e^{-\Omega(n^{0.09})}.
	\]

         Hence, it follows from union bound that, with probability at
         least $1-o(1)$, (i)-(iii) hold for any sequence $R^i$, $1\le
         i\le n^{1.1}$, satisfying $R_-^i\subseteq R^i\subseteq R_+^i$.
\end{proof}

In order to apply Lemma~\ref{pfrac}, we need an  additional requirement that
 the induced subgraphs $R_i$  be balanced.

\begin{lemma}\label{balanced}
	Let  $n,H,P,Q,S$ and $R_+^i, R_-^i$, $i\in [n^{1.1}]$,  be given as in Lemma~\ref{indsub}.
	Then, with probability $1-o(1)$, there exist subgraphs $R_i$,
        $i\in [n^{1.1}]$, such that
$R^i_-\subseteq R^i\subseteq R^i_+$ and $R^i$ is balanced.
\end{lemma}

\begin{proof}
	Recall that $|P|=n$, $|Q|=n/3$, $|S\cap P|=n^{0.99}$, and $|S\cap Q|=n^{0.99}/3$, and that
$R_+^i$ is formed by taking each vertex of $H$ independently and
uniformly at random  with probability $n^{-0.9}$.
So for $i\in [n^{1.1}]$,
	\begin{align*}
		& \E(|R^i_+\cap P|)=n^{0.1}, \\
        & \E(|R^i_+\cap P\cap S|)=n^{0.09},  \\
		& \E(|R^i_+\cap Q|)=n^{0.1}/3, \quad \text{and} \\
        & \E(|R^i_+\cap P\cap S|)=n^{0.09}/3.
	\end{align*}
	By Chernoff's inequality,
	\begin{align*}
		& \P\left(\large||R^i_+\cap P|-n^{0.1}\large|\ge n^{0.08}\right)\le e^{-\Omega(n^{0.06})}, \quad \\
		& \P\left(\large||R^i_+\cap P\cap S|-n^{0.09}\large|\ge n^{0.08}\right)\le e^{-\Omega(n^{0.07})}, \quad \\
		& \P\left(\large||R^i_+\cap Q|-n^{0.1}/3\large|\ge n^{0.08}\right)\le e^{-\Omega(n^{0.06})}, \quad \text{and} \\
		& \P\left(\large||R^i_+\cap Q\cap S|-n^{0.09}/3\large|\ge n^{0.08}\right)\le e^{-\Omega(n^{0.07})}.
	\end{align*}

	Thus, with probability $1-o(1)$, for all $i\in [n^{1.1}]$,
	\begin{align*}
		& |R^i_+\cap P|\in [n^{0.1}-n^{0.08},n^{0.1}+n^{0.08}],\\
        & |R^i_+\cap P\cap S|=(1+o(1))n^{0.09},  \\
		& |R^i_+\cap Q|\in [n^{0.1}/3-n^{0.08},n^{0.1}/3+n^{0.08}], \quad \text{and} \\
        & |R^i_+\cap Q\cap S|=(1+o(1)) n^{0.09}.
	\end{align*}
	     Therefore, 
\[
            \left||R^i_+\cap P|-3|R^i_+\cap Q|\right|\le
                4n^{0.08}<\min\{|R^i_+\cap P\cap S|, |R^i_+\cap Q\cap S|\}.
\]
	Hence, with probability $1-o(1)$, $R^i$ can be taken to be balanced for all $i\in [n^{1.1}]$.
\end{proof}

 Another smaller difference between here and \cite{AFH12} is that condition (ii) in Lemma~\ref{indsub}
 is slightly weaker than the corresponding condition in \cite{AFH12}. In
 \cite{AFH12} all vertices have almost the same degree, but here a small portion of the vertices could have smaller degree.
 The following lemma reflects a slightly weaker conclusion due to this difference, and the proof mainly follows that of Claim 4.1 in \cite{AFH12}.


\begin{lemma}\label{secrd}
	Let  $n$, $H,S$, $R^i$, $i=1,\dots, n^{1.1}$ be given as in
        Lemma~\ref{balanced} such that each $H[R^i]$ is a balanced
        $(1,3)$-partite $4$-graph and has a perfect fractional  matching $w^i$.
	Then there exists a spanning subgraph $H''$ of $H':=\bigcup_{i=1}^{n^{1.1}}H[R^i]$ such that
         \begin{itemize}
            \item [$(i)$] $d_{H''}(u)\le (1+o(1))n^{0.2}$ for $u\in S$,
            \item [$(ii)$] $d_{H''}(v) =(1+o(1))n^{0.2}$ for $v\in V(H)\setminus S$, and
            \item [$(iii)$] $\Delta_2(H'')\le n^{0.1}$.
         \end{itemize}
\end{lemma}

\begin{proof}
	Let $H'=\bigcup_{i=1}^{n^{1.1}} H[R^i]$.
                             By $(iid)$ of Lemma~\ref{indsub}, each edge of $H$ is contained in at
 most one $R^i$. Let $i_e$ denote the index $i$ such that $e\subseteq
 R^i$ (if exists); and let $w^{i_e}(e)=0$ when $i_e$ is not defined.
 	Let $H''$ be a spanning subgraph of $H'$ obtained by
        independently selecting each edge $e$ at random
    with probability $w^{i_e}(e)$.

	For $v\in V(H'')$, let $I_v=\{i:\ v\in R^i\}$, $E_v=\{e\in H':\ v\in e\}$, and $E_v^i=E_v\cap H[R^i]$.
    Then $E_v^i$, $ i\in I_v$, form a partition of $E_v$. Hence, for $v\in V(H'')$,
	\[
		d_{H''}(v)= \sum _{e\in E_v} 1 =\sum_{i\in I_v}\sum_{e\in E_v^i} X_e,
	\]
	where $X_e\sim Be(w^{i_e}(e))$ is the Bernoulli random variable with $X_e=1$ if $e\in E(H'')$ and $X_e=0$ otherwise.  Thus,
    since $\sum_{e\in E_v^i} w^{i}(e)=1$ (as $w^i$ is a perfect fractional  matching in $H[R^i]$),
	\[
		\E (d_{H''}(v))=\sum_{i\in I_v}\sum_{e\in E_v^i} w^{i}(e)=\sum_{i\in I_v} 1.
	\]
	Hence, $\E (d_{H''}(v))=(1+o(1))n^{0.2}$ for $v\in
        V(H)\setminus S$ (by $(iib)$ of Lemma~\ref{indsub}),
    and  $\E (d_{H''}(v))\le (1+o(1))n^{0.2}$ for $v\in S$ (by $(iia)$ of Lemma~\ref{indsub}).
	Now by Chernoff's inequality, for $v\in V(H)\setminus S$,
	\[
		\P(|d_{H''}(v)-n^{0.2}|\ge n^{0.15})\le e^{-\Omega(n^{0.1})},
	\]
	and for $v\in S$,
	\[
		\P(d_{H''}(v)-n^{0.2}\ge n^{0.15})\le e^{-\Omega(n^{0.1})}.
	\]
	
    Thus by taking union bound over all $v\in V(H)$, we have that, with probability $1-o(1)$, $d_{H''}(v)=(1+o(1))n^{0.2}$ for all $v\in V(H)\setminus S$ and $d_{H''}(v)\le(1+o(1))n^{0.2}$ for all $v\in S$.
	
	Next,  note that for distinct $u,v\in V(H)$,
	\[
		d_{H''}(\{u,v\})= \sum _{e\in E_u\cap E_v\cap E(H'')} 1=\sum_{i\in I_u\cap I_v}\sum_{e\in E_u^i\cap E_v^i} X_e
	\]
	and $\E (d_{H''}(\{u,v\}))= \sum_{i\in I_u\cap I_v}\sum_{e\in E_u^i\cap E_v^i} w^i(e)$.
    By (iic) in Lemma~\ref{indsub}, $|I_u\cap I_v|\le 2$. So $\E (d_{H''}(\{u,v\})) \le |I_u\cap I_v|\le 2$.
	Thus by Chernoff's inequality,
	\[
		\P(d_{H''}(\{u,v\})\ge n^{0.1})\le e^{-\Omega(n^{0.2})}.
	\]
	Hence by a union bound $\Delta_2(H'')\le n^{0.1}$ with probability $1-o(1)$.

    Therefore, with probability $1-o(1)$, $H''$ satisfies $(i)$,
    $(ii)$, and $(iii)$.
\end{proof}

We also need the following result attributed to Pippenger \cite{PS}, stated as Theorem 4.7.1 in \cite{AS}. A {\it cover} in a hypergraph $H$
is a set of edges whose union is $V(H)$.

\begin{lemma}[Pippenger and Spencer, 1989] \label{nibble}
	For every integer $k\ge 2$ and reals $r\ge 1$ and $a>0$, there are $\gamma=\gamma(k,r,a)>0$ and $d_0=d_0(k,r,a)$ such that for every $n$ and $D\ge d_0$ the following holds: Every $k$-uniform hypergraph $H=(V,E)$ on a set $V$ of $n$ vertices in which all vertices have positive degrees and which satisfies the following conditions:
	\begin{itemize}
	 	\setlength{\itemsep}{0pt}
		\setlength{\parsep}{0pt}
		\setlength{\parskip}{0pt}
		\item[$($1$)$] For all vertices $x\in V$ but at most $\gamma n$ of them, $d(x)=(1\pm \gamma)D$;
		\item[$($2$)$] For all $x\in V$, $d(x)<rD$;
		\item[$($3$)$] For any two distinct $x,y\in V$, $d(x,y)<\gamma D$;
	\end{itemize}
	contains a cover of at most $(1+a)(n/k)$ edges.
\end{lemma}

Note that $H$ contains a cover of at most $(1+a)(n/k)$ edges implies
that $H$ contains a matching of size at least $(1-(k-1)a)(n/k)$ (see, for example, \cite{PS}). Now we are ready to state and prove the main result of this section, which will
be used to find an almost perfect matching after deleting an absorber.

\begin{lemma}\label{almPM}
	Let $\sigma>0$ and $0<\rho\le \varepsilon/4\ll 1$, let $n$ be
        a sufficiently large positive integer, and
	let $H$ be a $(1,3)$-partite $4$-graph with partition classes
        $Q,P$ such that $3|Q|=|P|=n$.
	Suppose  $H$ is not $\varepsilon$-close to any $H_{1,3}(n,n/3)$ and
     $d_H(\{u,v\})\geq {n-1\choose 2}-{2n/3\choose 2}-\rho n^2$ for
     all $v\in Q$ and $u\in P$.
	Then $H$ contains a matching covering all but at most $\sigma n$ vertices.
\end{lemma}

\begin{proof}
     By Lemmas~\ref{indsub} and \ref{balanced}, we have the  random
     subgraphs $R^i$, $i\in [n^{1.1}]$, such that,
     with probability $1-o(1)$, all $R^i$
     satisfies the properties in Lemmas~\ref{indsub} and
     ~\ref{balanced}. In particular, $H[R_i]$ is balanced with respect
     to the partition classes $Q,P$.

     Next, by Lemma~\ref{dense}, $H$ is $(\F, \varepsilon/6))$-dense,
     where
\[
     \F=\{A\subseteq V(H) : |A\cap Q|\ge (1/3-\varepsilon/8) n \mbox{
       and } |A\cap P|\ge (2/3-\varepsilon/8) n\}.
\]
Note that
	\[
		e(H)= \sum_{v\in Q}\sum_{u\in P} d_H(\{u,v\})/3\ge (n/3)(n/3)\left(\binom{n-1}{2}-\binom{2n/3}{2}-\rho n^2\right)\ge n^4/100.
	\]
	Hence by Lemma~\ref{indep} (and choosing suitable $\alpha_1,\alpha_2,
        \gamma$), we see that, with probability $1-o(1)$, for all
        $i\in [n^{1.1}]$ and for all independent sets $J$ in $H[R^i]$,
    $|J\cap P|\le (\alpha_1 +\gamma+o(1))n^{0.1}<n/3-\varepsilon^2n$
    or $|J\cap Q|\le (\alpha_2
    +\gamma+o(1))n^{0.1}<2n/3-\varepsilon^2n$

	Moreover,  by (iii) of Lemma~\ref{indsub}, with probability $1-o(1)$,
	$d_{H[R^i]}(\{u,v\})>\binom{|R^i\cap P|-1}{2}-\binom{2|R^i\cap P|/3}{2}-\xi |R^i\cap P|^2$ for all $u\in P$ and $v\in Q$.
    Hence, by Lemma~\ref{pfrac}, $H[R^i]$ contains a perfect fractional  matching for all $i\in [n^{1.1}]$.

	Thus by Lemma~\ref{secrd}, there exists a spanning subgraph $H''$ of $\bigcup_{i=1}^{n^{1.1}}H[R^i]$ such that $d_{H''}(u)\le (1+o(1))n^{0.2}$ for each $u\in S$,
   $d_{H''}(v) =(1+o(1))n^{0.2}$ for each $v\in V(H)\setminus S$, and
   $\Delta_2(H'')\le n^{0.1}$. Hence, by Lemma~\ref{nibble} (by
   setting $D=n^{0.2}$),
   $H''$ contains a cover of at most $(1+a)(n/3)$ edges, where $a$ is
   a constant satisfying $0<a<\sigma/3$.

    Now by greedily deleting intersecting edges, we obtain a matching
    of size at least $(1-3a)(n/3)$. Hence $H$ contains a matching
    covering all but at most $\sigma n$, provided $n$ is sufficiently large.
\end{proof}

\section{Conclusion}

\begin{proof}[Proof of Theorem~\ref{prefect-(1,3)}]
	By Lemma~\ref{close}, we may assume $H$ is not
        $\varepsilon$-close to any $H_{1,3}(n,n/3)$, where $\varepsilon\ll 1$.
		By Lemma~\ref{absorbing}, $H_{1,3}({\cal F})$ has  a matching $M'$ such that, for some $0<\rho'\ll \rho \ll \varepsilon$,
      $|M'|\leq \rho n/4$ and,
     for any  $S\subseteq V(H_{1,3}({\cal F}))$ with $|S|\leq \rho' n$
     and $3|S\cap Q|=|S\cap P|$, $H_{1,3}({\cal F})[S\cup V(M)]$ has a perfect matching.

	  Let $H_1= H-V(M')$. Then
     $d_{H_1}(\{u,v\})\geq {n-1\choose 2}-{2n/3\choose 2}-\rho n^2$ for all $v\in Q\cap V(H_1)$ and $u\in P\cap V(H_1)$, and
      $H_1$ is not $(2\varepsilon)$-close to $H_{1,3}(n,n/3)$.

	By Lemma~\ref{almPM}, $H_1$ contains a matching $M_1$ covering
        all but at most $\sigma n$ vertices, where we
    choose $\sigma$ so that $0<\sigma<\rho'$.
	Now	$H[(V(H_1)\setminus V(M_1))\cup V(M)]$ has a perfect matching $M_2$. Clearly, $M_1\cup M_2$ forms a perfect matching in $H$.
\end{proof}


\newpage

\begin{thebibliography}{99}

\bibitem{AH09} R. Aharoni and E. Berger, Rainbow matchings in $r$-partite $r$-graphs, \emph{Electron. J. Combin.}, \textbf{16} (2009), \#R119.

\bibitem{AH17}  R. Aharoni and D. Howard, A rainbow $r$-partite version of the Erd\H os-Ko-Rado theorem, \emph{Combin. Prob. and Comp.}, \textbf{26} (2017), 321--337.

\bibitem{AH}  R. Aharoni and D. Howard, Size conditions for the
existence of rainbow matchings, Preprint.


\bibitem{AHS12} N. Alon, H. Huang, and B. Sudakov, Nonnegative $k$-sums, fractional
covers, and probability of small deviations, {\it J. Combin. Theory
  Ser. B},  {\bf 102} (2012), 784--796.

\bibitem{AFH12} N. Alon, P. Frankl, H. Huang, V. R\"{o}dl, A. Rucinski, and
  B. Sudakov, Large matchings in uniform hypergraphs and the conjectures of
Erd\H{o}s and Samuels, {\it  J. Combin. Theory Ser. A},  {\bf 119} (2012), 1200--1215.

\bibitem{AS}
N. Alon and J. Spencer, The Probabilistic Method, Fourth Edition, 2015.

\bibitem{BMS15}
J. Balogh, R. Morris, and W. Samotij, Independent sets in hypergraphs,
{\it  J. Amer. Math. Soc.}, {\bf 28} (2015), 669--709.

\bibitem{BDE76} B. Bollob\'as, D.E. Daykin, and P. Erd\H{o}s, Sets of independent edges of a hypergraphs, \emph{Quart. J. Math. Oxford Ser.}, \textbf{27} (1976), 25--32.

\bibitem{Erdos65}P. Erd\H{o}s,   A problem on independent $r$-tuples, \emph{Ann. Univ. Sci. Budapest. E\"otv\"os Sect. Math.}, \textbf{8} (1965),
93--95.


\bibitem{Fr13} P. Frankl, Improved bounds for Erd\H{o}s'
 matching conjecture, {\it J. Combin. Theory Ser. A}, {\bf 120} (2013), 1068--1072.

\bibitem{Fr17}
P. Frankl, On the maximum number of edges in a hypergraph with given matching number, \emph{Discrete Appl. Math.}, \textbf{216} (2017), 562--581.

\bibitem{FrK19+}P. Frankl and A. Kupavskii, The Erd\"os Matching Conjecture and Concentration Inequalities, Arxiv: 1806.08855v2.


\bibitem{FLM}
P. Frankl, T. {\L}uczak and K. Mieczkowska, On matchings in hypergraphs, \emph{Electron. J. Combin.}, \textbf{19} (2012), \#R42.

\bibitem{FR}
P. Frankl and V. R\"odl, Near perfect coverings in graphs and hypergraphs, \emph{Europ. J. Combin.} \textbf{6} (1985), 317--326.

\bibitem{Fu}
Z. F\"uredi, Matchings and covers in hypergraphs, \emph{Graphs and Combinatorics,} \textbf{4} (1988), 115--206.

\bibitem{HLS} H. Huang, P. Loh, and B. Sudakov, The size of a hypergraph and its matching number, \emph{Combinatorics, Probability and Computing,} \textbf{21}  (2012), 442--450.


\bibitem{Ka72} R. M. Karp, Reducibility among combinatorial problems,
{\it  Complexity of Computer Computations} (R. E. Miller;
J. W. Thatcher; J. D. Bohlinger (eds.)) (1972) New York: Plenum. pp. 85--103.

\bibitem{Kh13} I. Khan, Perfect matchings in 3-uniform hypergraphs with large vertex degree, {\it SIAM J. Discrete Math.}, {\bf 27} (2013), 1021--1039.

\bibitem{Kh16} I. Khan, Perfect matchings in 4-uniform hypergraphs,
  {\it J. Combin. Theory, Ser. B}, {\bf 116} (2016) 333--366.

\bibitem{KOT13} D. K\"uhn, D. Osthus, and A. Treglown, Matchings in 3-uniform hypergraphs, \emph{J. Combin. Theory, Ser. B},
\textbf{103} (2013), 291--305.

\bibitem{LYY}
H. Lu, X. Yu, and X. Yuan, Nearly perfect matchings in uniform
hypergraphs,  arXiv 1911.07431.

\bibitem{LM}
    T.  {\L}uczak and K. Mieczkowska. On Erd\H{o}s extremal problem on matchings in hypergraphs, \emph{J. Combin. Theory Ser. A},
\textbf{124} (2014), 178--194.

\bibitem{MN}
M. Matsumoto and N. Tokushige, The exact bound in the Erd\H{o}s-Ko-Rado theorem for cross-intersecting families, \emph{J. Combin.
Theory Ser. A}, \textbf{52} (1989), 90--97.

\bibitem{PS}
N. Pippenger and J. Spencer, Asymptotic behaviour of the chromatic index for hypergraphs, \emph{J. Combin. Theory, Ser. A}, \textbf{51} (1989), 24--42.

\bibitem{PY}
L. Pyber, A new generalization of the Erd\H{o}s-Ko-Rado theorem, \emph{J. Combin. Theory
Ser. A}, \textbf{43} (1986), 85--90.

\bibitem{RRS06} V. R\"odl, A. Ruci\'nski, and E. Szemer\'edi, Perfect matchings in uniform hypergraphs with large minimum degree,
\emph{European J. Combin.}, \textbf{27} (2006), 1333--1349.


\bibitem {ST15} D. Saxton and A. Thomason, Hypergraph containers, {\it Invent. Math.} {\bf 201} (2015) 925--992.

\bibitem{TZ13} A. Treglown and Y. Zhao, Exact minimum degree thresholds for perfect matchings in uniform hypergraphs II,
\emph{J. Combin. Theory Ser. A}, \textbf{120} (2013),  1463--1482.
\end{thebibliography}
\end{document}